\documentclass[a4paper]{article}

\usepackage[english]{babel}
\usepackage[T1]{fontenc}

\usepackage[a4paper,top=3cm,bottom=2cm,left=3cm,right=3cm,marginparwidth=1.75cm]{geometry}

\usepackage[export]{adjustbox}

\usepackage{bibspacing}

\usepackage{authblk}
\usepackage{wrapfig}
\usepackage{subcaption}
\usepackage[cal=boondox]{mathalfa}
\usepackage{marginnote}

\usepackage{mathrsfs}
\usepackage{amsmath}
\usepackage{amssymb}
\newcommand{\rom}[1]{%
  \textup{\uppercase\expandafter{\romannumeral#1}}%
}

\usepackage[utf8]{inputenc}
\usepackage{amsthm}
\usepackage{thmtools, thm-restate}
\declaretheorem{claim}
\declaretheorem{theorem}

\declaretheorem{corollary}
\newtheorem*{theorem*}{Theorem}

\usepackage{booktabs, adjustbox}
\usepackage{multirow}
\usepackage{threeparttable}
\usepackage{graphicx}
\usepackage{array}
\usepackage{tabularx}

\usepackage[colorlinks=true, allcolors=blue, pdfencoding=auto]{hyperref}

\newcommand{\bz}{{\pmb z}}
\newcommand{\bw}{{\pmb w}}
\newcommand{\bbw}{\bar{{\pmb w}}}
\newcommand{\hbw}{\widehat{{\pmb w}}}
\newcommand{\Adag}{A^{\dagger}}

\newcommand{\smin}{\sigma_{\mathsf{min}}}
\newcommand{\smax}{\sigma_{\mathsf{max}}}

\newcommand{\Rng}{\mathsf{Range}}
\newcommand{\Null}{\mathsf{Null}}

\newcommand{\bv}{{\pmb v}}

\newcommand{\mineig}{{\lambda_{\mathsf min}}}

\newcommand{\maxeig}{{\lambda_{\mathsf max}}}

\let\oldphi\phi
\renewcommand\phi{\operatorname{\oldphi}}

\newcommand{\tL}{\tilde{\mathcal{L}}}
\newcommand{\tell}{\tilde{\ell}}

\newcommand\defeq{\triangleq}

\usepackage{mathtools}

\usepackage{mathrsfs}
\DeclareMathAlphabet{\mathpzc}{OT1}{pzc}{m}{it}

\theoremstyle{definition}

\usepackage{mathtools}
\usepackage{listings}
\usepackage{tikz}
\usepackage[nomessages]{fp}
\usepackage{adjustbox}
\usepackage{outlines}
\usepackage{times}
\usepackage{graphicx,color}
\usepackage{array,float}
\usepackage{url}
\usepackage{amstext,amssymb,amsmath}
\usepackage{hyphenat}
\usepackage{amsthm}
\usepackage{verbatim}
\usepackage{bm}
\usepackage{paralist}
\usepackage{ulem}

\newtheorem{lem}{Lemma}[section]
\newtheorem{defn}[lem]{Definition}
\newtheorem{assumption}{Assumption}

\newcommand{\cL}{\mathcal{L}}

\newcommand{\ex}[2]{\underset{#1}{\mathbb{E}}\left[ #2 \right]}

\newcommand{\re}{\mathbb{R}}

\newcommand{\cH}{\mathcal{H}}

\usepackage{etoolbox}

\newcommand{\surround}[2][r]%
  {\ifstrequal{#1}{round}%
    {\left( #2 \right)}%
    {\ifstrequal{#1}{square}%
      {\left[ #2 \right]}%
      {\ifstrequal{#1}{curly}%
        {\left\{ #2 \right\}}%
        {\ifstrequal{#1}{angle}%
          {\left\langle #2 \right\rangle}%
          {\ifstrequal{#1}{|}%
            {\left\lvert #2 \right\rvert}%
            {\ifstrequal{#1}{||}%
              {\left\lVert #2 \right\rVert}%
              {\ifstrequal{#1}{floor}%
                {\left\lfloor #2 \right\rfloor}%
                {\ifstrequal{#1}{ceil}%
                  {\left\lceil #2 \right\rceil}%
                  {\ifstrequal{#1}{.}%
                    {\left. #2 \right.}%
                    {\left( #2 \right)}%
                  }%
                }%
              }%
            }%
          }%
        }%
      }%
    }%
  }

\title{On exponential convergence of SGD in non-convex over-parametrized learning}

 \author[]{Raef Bassily}
 \author[]{Mikhail Belkin}
 \author[]{Siyuan Ma}
 \affil{Department of Computer Science and Engineering}
 \affil{The Ohio State University}
 \affil{
\textit{ \textit{bassily.1@osu.edu},
\{mbelkin, masi\}@cse.ohio-state.edu}
 }

\begin{document}

\maketitle

\begin{abstract}

Large over-parametrized models learned via stochastic gradient descent (SGD) methods have become a  key element in modern machine learning. 
Although SGD methods are very effective in practice, most theoretical analyses of SGD suggest slower 
convergence than what is empirically observed. In our recent work~\cite{MBB17} we analyzed how interpolation, common in modern over-parametrized learning, results in exponential convergence of SGD with constant step size for convex loss functions. In this note, we extend those results to  a much broader non-convex function class satisfying the Polyak-Lojasiewicz (PL) condition. A number of important non-convex problems in machine learning, including some classes of neural networks, have been recently shown to satisfy the PL condition. 
We argue that the PL condition provides a relevant and attractive setting for many machine learning problems, particularly in the over-parametrized regime.

\end{abstract}


\section{Introduction}

 Stochastic Gradient Descent and its variants have become a staple of the algorithmic  foundations of  machine learning.  Yet many of its properties are not  fully understood, particularly in non-convex settings common in modern practice.  

In this note, we study convergence of Stochastic Gradient Descent (SGD) for the class of functions satisfying the Polyak-Lojasiewicz (PL) condition. This class contains all strongly-convex functions as well as a broad range of non-convex functions including those used in machine learning applications (see the discussion below). 

The primary purpose of this note is to show that in the interpolation setting (common in modern over-parametrized machine learning and studied in our previous work~\cite{MBB17}) SGD with fixed step size has exponential convergence for the  functions satisfying the PL condition. To the best of our knowledge, this is the first such exponential convergence result for a class of non-convex functions. 



Below, we discuss and highlight a number of aspects of the PL condition which differentiate it from the convex setting and make it more relevant to the practice and requirements of many machine learning problems.
We first recall that in the interpolation setting, a minimizer $\bw^*$ of the empirical loss $\cL(\bw)=\frac{1}{n}\sum_{i=1}\ell_i(\bw)$ satisfies that $\ell_i(\bw^*)=0$ for all $i$. We say that $\cL$ satisfies the PL condition (see~\cite{KNS18}) if $\|\nabla \cL(\bw)\|^2\geq \alpha \cL(\bw)$  for some $\alpha>0$.


Most analyses for optimization in machine learning have concentrated on convex or, commonly,  strongly convex setting. These settings are amenable to theoretical analyses and describe many important special cases of ML, such as linear and kernel methods. Still, a large class of modern models, notably neural networks, are non-convex. Even for kernel machines, many of the arising optimization problems are poorly conditioned and not well-described by the traditional strongly convex analysis. Below we list  some properties of the PL-type setting which make it particularly attractive and relevant to the requirements of machine learning, especially in the interpolated and over-parametrized setting.  
\begin{itemize}
    \item[\it Ease of verification.] To verify the PL condition in the interpolated setting we need access to the norm of the gradient $\|\nabla \cL(\bw)\|$ and the value of the objective function $\cL(\bw)$. These quantities are typically easily accessible empirically\footnote{In general we need to evaluate $\cL(\bw)- \cL(\bw^*)$. Since $\cL(\bw^*)=0$, no further knowledge about $\bw^*$ is required.}, can be accurately estimated from a sub-sample of the data, and are often tractable analytically. 
    On the other hand, verifying convexity requires the  cumbersome  positive definiteness of the Hessian matrix requiring accurate estimation of its {\it smallest} eigenvalue $\lambda_\mathsf{min}$. Verifying this empirically is often difficult and cannot always be based on a sub-sample due to the required precision of the estimator when $\lambda_\mathsf{min}$ is close to zero (as is frequently the case in practice). 
    
    \item[\it Robustness of the condition.] The norm of the gradient is much more resilient to  perturbation of the objective function than the smallest eigenvalue of the Hessian (for convexity). 
    
    \item[\it Admissibility of multiple global minima.] Many modern machine learning methods are over-parametrized and result in manifolds of global minima~\cite{cooper2018loss}. This is not compatible with strict convexity and, in most circumstances\footnote{Unless those manifolds are convex domains in lower-dimensional affine sub-spaces.}, not compatible with convexity. However, manifolds of solutions are compatible with the PL condition. 
     
    \item[\it Invariance under transformations.] Nearly every application of machine learning  employs techniques for feature extraction or feature transformation. Global minima and the property of interpolation (shared global minima for the individual loss functions) are preserved under coordinate transformations. Yet convexity is generally not, thus not allowing for a unified analysis of optimization under feature transforms.
    In contrast, as discussed in Section~\ref{sec:transf-inv}, the PL condition is invariant under a broad class of non-linear coordinate transformations.
    
    \item[\it PL on manifolds.] Many problems of interest in machine learning involve optimization on manifolds. While geodesic convexity allows for efficient optimization, it is a parametrization dependent notion and is generally difficult to establish, as it requires explicit knowledge of the geodesic coordinates on the manifold. In contrast, the PL condition also allows for efficient optimization, while invariant under the choice of coordinates and far easier to verify. See~\cite{weber2017frank} for some recent applications.
    

    \item[\it Convergence analysis independent of the distance to the minimizer.] Most  convergence analyses in convex optimization rely on the distance to the minimizer. Yet, this distance is often difficult or impossible to  bound empirically. Furthermore, the distance to minimizer can be infinite in many important settings, including optimization via logistic loss~\cite{soudry2017implicit} or inverse problems over Hilbert spaces, as in kernel methods~\cite{ma2017diving}. In contrast, PL-type analyses directly involve {\it the value} of the loss function, an empirically observable quantity of practical significance. 
    
    \item[\it Exponential convergence of GD and SGD.] As originally observed by Polyak~\cite{polyak1963gradient}, the PL condition is sufficient for exponential convergence of gradient descent.  As we establish in this note, it also allows for exponential convergence of stochastic gradient descent with fixed step size in the interpolated setting. 
\end{itemize}

\noindent{\bf Technical contributions:} 
The main technical contribution of this note is to show the exponential convergence of mini-batch SGD in the interpolated setting. The proof is simple and is reminiscent of the original observation by Polyak~\cite{polyak1963gradient} of exponential convergence of gradient descent. It also extends our previous work on the exponential convergence of mini-batch SGD~\cite{MBB17} to a non-convex setting. Interestingly, the step size arising from the PL condition in our analysis depends on the parameter $\alpha$ and is potentially much smaller  than that in the strongly convex case, where no such dependence is needed. At this point it is an open question whether this dependence is necessary in the PL setting. As an additional contribution, in Section~\ref{sec:conv}, we show that for a special class of PL functions obtained by a composition of a strictly convex function and a linear transformation\footnote{These functions are convex but not necessarily strictly convex.}, we obtain exponential convergence without such dependence on $\alpha$ in the step size. However, this result requires a different type of analysis than that for the general PL setting. 
In Section~\ref{sec:transf-inv}, we provide a formal statement capturing the transformation invariance property of the PL condition.  

\paragraph{Examples and Related Work:}
The PL condition has recently become popular in optimization and machine learning starting with the work~\cite{KNS18}.
In fact, as discussed in~\cite{KNS18}, several other conditions proposed for convergence analysis are special cases of the PL condition. One such condition is Restricted secant inequality (RSI) proposed in \cite{zhang2013gradient}. 
Another set of conditions that are special cases of the PL condition was referred to as ``one-point convexity'' in \cite{allen2017natasha}. 
The two variations of one-point convexity discussed there are special cases of RSI and PL, respectively, and hence are in the PL class. The same reference points out several examples of ``one-point convexity'' in previous works. 
Some notable examples satisfying RSI include two-layer neural networks~\cite{li2017convergence}, matrix completion~\cite{sun2016guaranteed}, dictionary learning~\cite{arora2015simple}, and phase retrieval~\cite{chen2015solving}. It has also been observed empirically that neural networks satisfy the PL condition~\cite{kleinberg2018alternative}.
In particular, we note the recent work~\cite{soltanolkotabi2018theoretical} which considers a class of neural networks that attain  zero quadratic loss implying interpolation. In their proof it is shown that this class of neural nets satisfies the PL condition. Hence our results imply exponential convergence of SGD for this class. To the best of our knowledge this is the first time that exponential convergence of SGD has been established for a class of multi-layer neural networks.

\section{Exponential Convergence of SGD for PL Losses}

We start by formally stating the Polyak- Lojasiewicz (PL) Condition. 

\begin{defn}[$\alpha$-PL function]\label{defn:PL}
Let $\alpha >0$. Let $f:\cH\rightarrow \re$ be a differentiable function. Assume, w.o.l.g., that $ \inf\limits_{v\in\cH}f(v)=0$.  We say that $f$ is $\alpha$-PL if for every $w\in \cH$, we have 
$$\|\nabla f(w)\|^2\geq \alpha f(w).$$
\end{defn}

\paragraph{ERM with smooth losses:} We consider the ERM problem where for all $1\leq i\leq n$, $\ell_i$ is $\beta$-smooth. Moreover, $\cL(\bw)= \frac{1}{n}\sum_{i=1}^n\ell_i(\bw)$ is 
$\lambda$-smooth, 
$\alpha$-PL function (as in Definition~\ref{defn:PL} above).

We do not assume compact parameter space; that is, a parameter vector $\bw\in\cH$ can have unbounded norm, however $\cL$ is assumed to be bounded. In particular, a global minimizer may not exist, however, we assume the existence of global infimum for $\cL$ (which is equal to zero w.o.l.g.). 


To elaborate, we assume the existence of a sequence $\bw_1, \bw_2, \ldots$ such that 
\begin{align}
\lim\limits_{k\rightarrow\infty}\cL(\bw_k)=\inf\limits_{\bw\in\cH}\cL(\bw)=0\label{eq:L_min}
\end{align}



\begin{assumption}[Interpolation]\label{assump:interp} 
For every sequence $\bw_1, \bw_2, \ldots$ such that $\lim\limits_{k\rightarrow \infty}\cL(\bw_k)=0$, we have for all $1\leq i\leq n$, $\lim\limits_{k\rightarrow\infty}\ell_i(\bw_k)=0$. 
\end{assumption}

Consider the SGD algorithm that starts at an arbitrary $\bw_0\in\cH$, and at each iteration $t$ makes an update with a constant step size $\eta$:
\begin{align}
\bw_{t+1}&=\bw_t -\eta \cdot \nabla \left\{ \frac{1}{m}\sum_{j=1}^m  \ell_{i_t^{(j)}}(\bw_t)
\right\}
\label{main-update-step}
\end{align}
where $m$ is the size of a mini-batch of data points whose indices $\{i_t^{(1)}, \ldots, i_t^{(m)}\}$ are drawn uniformly with replacement at each iteration $t$ from $\{1, \ldots, n\}$.

The theorem below establishes the exponential convergence of mini-batch SGD for any smooth, PL loss $\cL$ in the interpolated regime. 

\begin{theorem}\label{thm:PL-SGD}
Consider the mini-batch SGD with smooth losses as described above. Suppose that Assumption~\ref{assump:interp} holds and suppose that the empirical risk function $\cL$ is $\alpha$-PL for some fixed $\alpha >0$. For any mini-batch size $m \in \mathbb{N}$, the mini-batch SGD~(\ref{main-update-step}) with \emph{constant} step size
$\eta^*(m) \defeq \frac{\alpha m}{\lambda\left(\beta + \lambda (m - 1)\right)}$
gives the following guarantee
\begin{equation}\label{eq:theorem1}
\ex{\bw_t}{\cL(\bw_t)}\leq
\left(1 - \frac{\alpha\,\eta^*(m)}{2}\right)^t\,\cL(\bw_0)
\end{equation}
where the expectation is taken w.r.t. the randomness in the choice of the mini-batch.
\end{theorem}

\begin{proof}
From $\lambda$-smoothness of $\cL$, it follows that
$$\cL(\bw_{t+1})\leq \cL(\bw_{t})+\langle \nabla\cL(\bw_t),~ \bw_{t+1}-\bw_t\rangle+\frac{\lambda}{2}\|\bw_{t+1}-\bw_{t}\|^2.$$
Using \ref{main-update-step}, we then have
$$\cL(\bw_{t})-\cL(\bw_{t+1})\geq \eta \left\langle~ \nabla\cL(\bw_t)~,~ \frac{1}{m}\sum_{j=1}^m  \nabla \ell_{i_t^{(j)}}(\bw_t)
~\right\rangle -\frac{\eta^2\lambda}{2}\left\| \frac{1}{m}\sum_{j=1}^m  \nabla \ell_{i_t^{(j)}}(\bw_t)\right\|^2.$$ 

Fixing $\bw_t$ and taking expectation with respect to the randomness in the choice of the batch $i_t^{(1)}, \ldots, i_t^{(m)}$ (and using the fact that those indices are i.i.d.), we get

$$\ex{i_t^{(1)},\ldots, i_t^{(m)}}{\cL(\bw_{t})-\cL(\bw_{t+1})}\geq \eta \left\|\nabla \cL(\bw_t)\right\|^2 - \eta^2 \frac{\lambda}{2}\left(\frac{1}{m}~\ex{i_t^{(1)}}{\|\nabla \ell_{i_t^{(1)}}(\bw_t)\|^2}+\frac{m-1}{m}\|\nabla \cL(\bw_t)\|^2\right)$$
Since $\forall i\in [n], \ell_i$ is $\beta$-smooth and non-negative, we have $\|\nabla \ell_{i_t^{(1)}}(\bw_t)\|^2 \leq 2\beta \ell_{i_t^{(1)}}(\bw_t)$ with probability $1$ over the choice of $i_t^{(1)}$. 
Thus, the last inequality reduces to
$$\ex{}{\cL(\bw_{t})-\cL(\bw_{t+1})}\geq \eta\,\left(1-\frac{\eta\lambda}{2}\frac{m-1}{m}\right) \|\nabla \cL(\bw_t)\|^2 - \eta^2 \frac{\lambda\beta}{m}\cL(\bw_t).$$

By invoking $\alpha$-PL condition of $\cL$ and assuming that $\eta \leq \frac{2}{\lambda}$, we get
\begin{align*}
\ex{}{\cL(\bw_{t})-\cL(\bw_{t+1})}&\geq \alpha\, \eta\,\left(1-\frac{\eta\lambda}{2}\frac{m-1}{m}\right)\cL(\bw_t)- \eta^2 \frac{\lambda\beta}{m}\cL(\bw_t)\\
&=\eta \left(\alpha - \eta\frac{\lambda}{m}\left(\alpha\frac{m-1}{2}+\beta\right)\right)\cL(\bw_t)
\end{align*}
Hence,
\begin{align}
\ex{}{\cL(\bw_{t+1})}&\leq \left(1-\eta\,\alpha + \eta^2\frac{\lambda}{m}\left(\alpha\frac{m-1}{2}+\beta\right)\right)\ex{}{\cL(\bw_t)}\label{eq:final_bound_PL-SGD}
\end{align}
By optimizing the quadratic term in the upper bound~\ref{eq:final_bound_PL-SGD} with respect to $\eta$, we get $\eta = \frac{\alpha m}{\lambda\left(\beta + \lambda (m - 1)\right)}$, which is $\eta^*(m)$ in the theorem statement. Hence, (\ref{eq:final_bound_PL-SGD}) becomes 
$$\ex{}{\cL(\bw_{t+1})}\leq
\left(1 - \frac{\alpha\,\eta^*(m)}{2}\right)\,\ex{}{\cL(\bw_t)},$$
which gives the desired convergence rate.
\end{proof}

\section{A Transformation-Invariance Property of PL Functions and Its Implications}\label{sec:transf-inv}

In this section, we formally discuss a simple observation concerning the class of PL functions that has useful implications on wide array of problems in modern machine learning. In particular, we observe that if $f: \cH\rightarrow \re$ is $\lambda$-smooth and $\alpha$-PL function for some $\lambda, \alpha>0$, then for any map $\Phi:\cH' \rightarrow \cH$ that satisfies certain weak conditions, the composition  $f\left(\Phi\left(\cdot\right)\right): \cH' \rightarrow \re$ is $\lambda'$-smooth and $\alpha'$-PL for some $\lambda', \alpha'>0$ that depend on $\lambda, \alpha$, respectively, as well as a fairly general property of $\Phi$. This shows that the class of smooth PL objectives is closed under a fairly large family of transformations. Given our results above, this observation has direct implications on the convergence of SGD for large class of problems that involve parameter transformation, e.g., via feature maps. 

First, we formalize this closure property in the following claim. Let $\Phi:\re^k\rightarrow\re^d$ be any map. We can write such a map as $\Phi=\left(\phi_1, \ldots, \phi_d\right)$, where for each $j\in [d],$ $\phi_j: \re^k \rightarrow \re$ is a scalar function over $\re^k$. The Jacobian of $\Phi$ is an operator $J_{\Phi}: \re^k \rightarrow \re^d$ that, for each $\bw=(w_1, \ldots, w_k)\in \re^k$, is described by a $d\times k$ real-valued matrix $J_{\Phi}(\bw)$ whose entries are the partial derivatives $\frac{\partial \phi_j}{\partial w_p},~ 1\leq j\leq d, ~1\leq p\leq k.$ 

\begin{claim}\label{claim:trans_inv}
Let $f: \re^d\rightarrow \re$ be $\lambda$-smooth and $\alpha$-PL function for some $\lambda, \alpha>0$. Let $\Phi: \re^k \rightarrow \re^d$ be any map, where $d\geq k$.  Suppose there exist $b \geq a >0$ such that for all $\bw\in \re^k,~$ $\mineig\left(J_{\Phi}(\bw)^T\,J_{\Phi}(\bw)\right)\geq a$ and $\maxeig\left(J_{\Phi}(\bw)^T\,J_{\Phi}(\bw)\right)\leq b $, where $\mineig\left(J_{\Phi}(\bw)^T\,J_{\Phi}(\bw)\right)$ and $\maxeig\left(J_{\Phi}(\bw)^T\,J_{\Phi}(\bw)\right)$ denote the minimum and maximum eigen values of $J_{\Phi}(\bw)^T\,J_{\Phi}(\bw)$, respectively. Then, the function $f\left(\Phi\left(\cdot\right)\right): \re^k \rightarrow \re$ is $\lambda'$-smooth and $\alpha'$-PL, where $\alpha' = a\, \alpha$ and $\lambda'= b\, \lambda$. 
\end{claim}

Note that the condition that $d\geq k$ is necessary for $\mineig\left(J_{\Phi}(\bw)^T\,J_{\Phi}(\bw)\right)$ to be positive. The condition on $\maxeig\left(J_{\Phi}^T(\bw)\,J_{\Phi}(\bw)\right)$ holds when $\Phi$ is differentiable and Lipschitz-continuous.  The above claim follows easily from the chain rule and the PL condition. 

Given this property of PL functions and our result in Theorem~\ref{thm:PL-SGD}, we can argue that for smooth, PL losses, the exponential convergence rate of SGD is preserved under any transformation that satisfies the conditions in the above claim. We formalize this conclusion below. 

As before, we consider a set of $\beta$-smooth losses $\ell_i:\re^d\rightarrow \re,~ 1\leq i\leq n$, where the empirical risk $\cL(\bw)=\frac{1}{n}\sum_{i=1}^n \ell_i(\bw)$ is $\lambda$-smooth and $\alpha$-PL. 

\begin{corollary}
Let $\Phi: \re^k\rightarrow \re^d$ be any map that satisfies the conditions in Claim~\ref{claim:trans_inv}. Suppose Assumption~\ref{assump:interp} holds and that there is sequence $\bw_1, \bw_2, \ldots \in \mathsf{Image}(\Phi)$ such that $\lim\limits_{j\rightarrow\infty} \cL(\bw_j)=\cL_{\mathsf{min}}=0$. Suppose we run mini-batch SGD w.r.t. the loss functions $\ell_i\left(\Phi\left(\cdot\right)\right), ~1\leq i\leq n,$ with batch size $m$ and step size $\eta_{\Phi}(m)=\frac{a}{b^2}\eta^*(m)$, where $\eta^*(m)$ is as defined in Theorem~\ref{thm:PL-SGD}. Let $\bv_0, \bv_1, \ldots, \bv_t \in \re^k$ denote the sequence of parameter vectors generated by mini-batch SGD over $t$ iterations. Then, we have 
$$\ex{\bv_t}{\cL\left(\Phi\left(\bv_t\right)\right)}\leq
\left(1 - \left(\frac{a^2}{b^2}\right)\frac{\alpha\,\eta^*(m)}{2}\right)^t\,\cL\left(\Phi(\bv_0)\right).$$
\end{corollary}

\section{Faster Convergence for a Class of Convex Losses}\label{sec:conv}

We consider a special class of PL functions originally discussed in \cite{KNS18}. This class contains all convex functions $f:\re^d\rightarrow \re$ that can be expressed as a composition $g\left(A(\cdot)\right)$ of a strongly convex function $g:\re^k\rightarrow \re$ with a linear function $A:\re^d\rightarrow\re^k$. Note that this class contains convex losses that are convex but not necessarily strongly, or even strictly convex. 

In \cite[Appendix B]{KNS18}, it was shown that if $g:\re^k\rightarrow \re$ is $\alpha$-strongly convex and $A\in \re^{k\times d}$ is matrix whose \emph{least non-zero} singular value is $\sigma$, then $f:\re^d\rightarrow \re$ defined as $f(\bw)\triangleq g\left(A\bw\right),~ \bw\in\re^d$ is $\alpha\, \sigma^2$-PL function. 
For this special class of PL losses, we show a better bound on the convergence rate than what is directly implied by Theorem~\ref{thm:PL-SGD}. The proof technique for this result is different from that of Theorem~\ref{thm:PL-SGD}. Exponential convergence of SGD for strongly convex losses in the interpolation setting has been established previously in \cite{MBB17}. In this section, we show a similar convergence rate for this larger class of convex losses.  

Let $A\in \re^{k\times d}$. Let $\sigma_{\mathsf{min}}$ and $\sigma_{\mathsf{max}}$ denote the smallest non-zero singular value and the largest singular value of $A$, respectively. Consider a collection of loss functions $\ell_i:\re^d\rightarrow\re, i=1, \ldots, n,$ where each $\ell_i$ can be expressed as $\ell_i(\bw)=\tell_i(A\bw)$ for some $\beta$-smooth convex function $\tell_i: \re^{k}\rightarrow\re$. It is easy to see that this implies that each $\ell_i$ is $\sigma_{\mathsf{max}}^2\beta$-smooth and convex. The empirical risk $\cL(\bw)=\frac{1}{n}\sum_{i=1}^n \ell_i(\bw)$ can be written as $\cL(\bw)=\tL(A\bw)\triangleq\frac{1}{n}\sum_{i=1}^n\tell_i(A\bw)$. Moreover, suppose that $\tL$ is $\lambda$-smooth and $\alpha$-strongly convex. Now, suppose we run SGD described in (\ref{main-update-step}) to solve the ERM problem defined by the losses $\ell_i, 1\leq i\leq n.$ The following theorem provides an exponential convergence guarantee for SGD in the interpolation setting. 

\begin{theorem}\label{thm:sgd-sc-lin}
Consider the scenario described above and suppose Assumption~\ref{assump:interp} is true. Let $\smin$ and $\smax$ be the smallest non-zero singular value and the largest singular value of $A$, respectively. Let $\bw^*\in\re^d$ be any vector such that $A\bw^*$ is the unique minimizer of $\tL$. The mini-batch SGD~(\ref{main-update-step}) with batch size $m$ and step size $\eta^*(m)=\frac{m}{\smax^2\left(\beta+(m-1)\lambda\right)}$ gives the following guarantee
\begin{align*}
    \ex{\bw_{t}}\cL(\bw_{t})&\leq \frac{\lambda\smax^2}{2}(1-\alpha\,\smin^2\,\eta^*(m))^t\|\hbw_0-\hbw^*\|
\end{align*}
where $\hbw_0=\Adag \bw_0$ ~and~ $\hbw^*=\Adag\bw^*$ where $\Adag$ is the pseudo-inverse of $A$.
\end{theorem}

\begin{proof}
Recall that we can express $A$ via SVD as $A=U\Sigma V^T$ where $U=[U_1 \ldots U_k]$ is the $k\times k$ matrix whose columns form an eigen basis for $AA^T$, $V=[V_1 \ldots V_d]$ is the $d\times d$ matrix whose columns form an eigen basis for $A^TA$, and $\Sigma$ is $k\times d$ matrix that contains the singular values of $A$; in particular $\Sigma_{ii}=\sigma_i$ and $\Sigma_{ij}=0$ for $i\neq j,~ 1\leq i\leq k, 1\leq j\leq d$, where $\sigma_i$ is the $ith$ singular value of $A$, $1\leq i\leq \min\{k, d\}$. Let $\smax\triangleq\sigma_1\geq \sigma_2\geq \ldots \geq \sigma_r\triangleq\smin$ be the non-zero singular values of $A$, where $r\leq \min\{k, d\}$. The following is a known fact: $\{U_1, \ldots, U_r\}$ is orthonormal basis for $\Rng(A)$ and $\{V_1, \ldots, V_r\}$ is orthonormal basis for $\Null(A)^{\perp}$, where $\Null(A)^{\perp}$ is the subspace orthogonal to $\Null(A)$. Also, recall that the Moore-Penrose inverse (pseudo-inverse) of $A$, denoted as $\Adag$ is given by $\Adag=V\Sigma^{\dagger}U^T$, where $\Sigma^{\dagger}$ where $\Sigma^{\dagger}_{ii}=\sigma_i^{-1},~ 1\leq i \leq r$, and the remaining entries are all zeros. The following is also a known fact that follows easily from the definition of $\Adag$ and the facts above: $\{V_1, \ldots, V_r\}$ is orthonormal basis for $\Rng(\Adag)$. Hence, from the above facts, it is easy to see that $\Rng(\Adag)=\Null(A)^{\perp}$. Thus, by the direct sum theorem, any $\bw\in \re^d$ can be uniquely expressed as sum of two orthogonal components $\hbw+\bar{\bw}$, where $\hbw\in\Rng(\Adag)$ and $\bbw\in\Null(A)$. In particular, $\hbw=\Adag A\bw$.  

Using these observations, we can make the following claim. 
\begin{claim}\label{claim:strng-conv-range}
$\cL$ is $\alpha\,\smin^2$-strongly convex over $\Rng(\Adag)$.
\end{claim}
The proof of the above claim is as follows. Fix any $\bz_1, ~\bz_2\in \Rng(\Adag)$. Observe that 
\begin{align}
    \cL(\bz_1)=\tL(A\bz_1)&\geq \tL(A\bz_2)+\left\langle \nabla \tL(A\bz_2), ~ A\left(\bz_1-\bz_2\right)\right\rangle+\frac{\alpha}{2}\left\|A\left(\bz_1-\bz_2\right)\right\|^2\label{ineq:tL-strng-conv}\\
    &=\cL(\bz_2)+\left\langle \nabla \cL(\bz_2), ~ \bz_1-\bz_2\right\rangle+\frac{\alpha}{2}\left\|A\left(\bz_1-\bz_2\right)\right\|^2\label{eq:cL}
\end{align}
where (\ref{ineq:tL-strng-conv}) follows from the strong convexity of $\tL$, and (\ref{eq:cL}) follows from the definition of $\cL$ and the fact that $\nabla \cL(\bz_2)=A^T \nabla \tL(A\bz_2)$. Now, we note that since $\bz_1, \bz_2 \in \Rng(\Adag)$, we have $\left\|A\left(\bz_1-\bz_2\right)\right\|^2 = \sum_{j=1}^r \sigma_j^2 \left\langle V_j, ~ \bz_1 - \bz_2~\right\rangle^2 \geq \smin^2 \left\|\bz_1-\bz_2\right\|^2$. Plugging this into (\ref{eq:cL}) proves the claim. 

We now proceed with the proof of the Theorem~\ref{thm:sgd-sc-lin}.
By $\lambda$-smoothness of $\tL$, we have 
\begin{align}
    \cL(\bw_{t+1})&=\tL(A\bw_{t+1})\leq \frac{\lambda}{2}\|A(\bw_{t+1}-\bw^*)\|^2\nonumber\\
    &=\frac{\lambda}{2}\|A(\hbw_{t+1}-\hbw^*)\|^2\leq \frac{\smax^2\lambda}{2}\|\hbw_{t+1}-\hbw^*\|^2.\label{bound_on_loss}
\end{align}
where, as above, $\hbw_{t+1}=\Adag A\bw_{t+1}$ is the projection of $\bw_{t+1}$ onto $\Rng(\Adag)$. Similarly, $\hbw^*$ is the projection of $\bw^*$ onto $\Rng(\Adag)$. Now, consider $\|\hbw_{t+1}-\hbw^*)\|^2$. From the update step (\ref{main-update-step}) of the mini-batch SGD and the linearity of the projection operator $\Adag A$, we have
\begin{align*}
    \|\hbw_{t+1}-\hbw^*\|^2 &=\|\hbw_t-\hbw^*)\|^2 -2\eta\, \left\langle \Adag A\cdot \frac{1}{m}\sum_{j=1}^m  \nabla \ell_{i_t^{(j)}}(\hbw_t)~,~ ~\hbw_t-\hbw^*\right\rangle+\eta^2 \left\|\Adag A\cdot\frac{1}{m}\sum_{j=1}^m  \nabla \ell_{i_t^{(j)}}(\hbw_t)\right\|^2\\
    &\leq \|\hbw_t-\hbw^*\|^2 -2\eta\, \left\langle  \frac{1}{m}\sum_{j=1}^m  \nabla \ell_{i_t^{(j)}}(\hbw_t), ~\hbw_t-\hbw^*\right\rangle+\eta^2 \left\|\frac{1}{m}\sum_{j=1}^m  \nabla \ell_{i_t^{(j)}}(\hbw_t)\right\|^2
\end{align*}
where the first equality follows from the update step and the fact that $\forall ~1\leq i\leq n, ~ \nabla\ell_i(\bw_t)=A^T\nabla\tell_i(A\bw_t) = A^T\nabla\tell_i(A\hbw_t)=\nabla\ell_i(\hbw_t)$. The last inequality follows from the fact that $\left(\mathbb{I}-\Adag A\right)\cdot\frac{1}{m}\sum_{j=1}^m  \nabla \ell_{i_t^{(j)}}(\hbw_t)$ is orthogonal to $\Rng(\Adag)$ (and hence orthogonal to $\hbw_t-\hbw^*$), and the fact that projection cannot increase the norm.
Fixing $\hbw_t$ and taking expectation with respect to the choice of the batch $i_t^{(1)}, \ldots, i_t^{(m)}$, we have 

\begin{align}
\ex{i_t^{(1)}, \ldots, i_t^{(m)}}{\|\hbw_{t+1}-\hbw^*\|^2}&\leq \|\hbw_t-\hbw^*\|^2 -2\eta\, \left\langle \nabla \cL(\hbw_t), ~\hbw_t-\hbw^*\right\rangle+\eta^2 ~~\ex{i_t^{(1)}, \ldots, i_t^{(m)}}{\left\|\frac{1}{m}\sum_{j=1}^m  \nabla \ell_{i_t^{(j)}}(\hbw_t)\right\|^2}\label{ineq:distance}
\end{align}
By Claim~\ref{claim:strng-conv-range}, we have 
\begin{align}
    \left\langle \nabla \cL(\hbw_t), ~\hbw_t-\hbw^*\right\rangle&\geq \cL(\hbw_t)+\frac{\alpha\,\smin^2}{2}\|\hbw_t-\hbw^*\|^2\label{ineq:strng-conv}
\end{align}
Hence, from (\ref{ineq:distance})-(\ref{ineq:strng-conv}), we have 
\begin{align*}
    \ex{i_t^{(1)}, \ldots, i_t^{(m)}}{\|\hbw_{t+1}-\hbw^*\|^2}&\leq \left(1-\eta\alpha\,\smin^2\right)~\ex{i_t^{(1)}, \ldots, i_t^{(m)}}{\|\hbw_t-\hbw^*\|^2}\\
    &\hspace{0.3cm}-2\eta \left(\cL(\hbw_t)-\frac{\eta}{2}~~\ex{i_t^{(1)}, \ldots, i_t^{(m)}}{\left\|\frac{1}{m}\sum_{j=1}^m  \nabla \ell_{i_t^{(j)}}(\hbw_t)\right\|^2}\right)
\end{align*}

As noted earlier $\forall~1\leq i \leq n,~ \ell_i$ is $\beta\smax^2$-smooth. Also, it is easy to see that$\cL$ is $\lambda\smax^2$-smooth. From this point onward, the proof follows the same lines of the proof of \cite[Theorem~1]{MBB17}. We thus can show that by choosing $\eta=\eta^*(m)=\frac{m}{\smax^2\left(\beta+(m-1)\lambda\right)}$, we get 
\begin{align*}
    \ex{i_t^{(1)}, \ldots, i_t^{(m)}}{\|\hbw_{t+1}-\hbw^*\|^2}&\leq \left(1-\eta^*(m)\alpha\,\smin^2\right)\ex{i_t^{(1)}, \ldots, i_t^{(m)}}{\|\hbw_t-\hbw^*\|^2}
\end{align*}
Using the above inequality together with (\ref{bound_on_loss}), we have
\begin{align*}
    \ex{\bw_{t+1}}\cL(\bw_{t+1})&\leq \frac{\smax^2\lambda}{2}\left(1-\eta^*(m)\alpha\,\smin^2\right)\ex{\bw_{t}}{\|\hbw_t-\hbw^*\|^2}\\
    &\leq \frac{\smax^2\lambda}{2}\left(1-\eta^*(m)\alpha\,\smin^2\right)^{t+1}\|\hbw_0-\hbw^*\|^2
\end{align*}
\end{proof}

\bibliography{sample}
\bibliographystyle{plain}
\end{document}